\documentclass[10pt, draft]{amsart}

\usepackage[cp1251]{inputenc}
\usepackage[T2A]{fontenc}
\usepackage{amsmath}
\usepackage{amsfonts}
\usepackage{amssymb}
\usepackage{amsthm}
\usepackage[ukrainian,russian,english]{babel}
\theoremstyle{plain}
\newtheorem{theorem}{Theorem}
\newtheorem{lemma}[theorem]{Lemma}

\newtheorem{corollary}{Corollary}[theorem]
\newtheorem{example}{Example}
\theoremstyle{definition}

\theoremstyle{remark}

\theoremstyle{plain}
\newtheorem*{theorem*}{Theorem}
\newtheorem*{lemma*}{Lemma}
\newtheorem*{proposition*}{Proposition}
\newtheorem*{statement*}{Statement}
\newtheorem*{corollary*}{Corollary}
\newtheorem*{example*}{Example}
\theoremstyle{definition}
\newtheorem*{definition*}{Definition}
\theoremstyle{remark}
\newtheorem*{notation*}{Notation}
\newtheorem*{remark*}{Remark}
\renewcommand{\theexample}%
{\Alph{example}}

\begin{document}
\footnotetext[1]{The investigation is partially supported by DFFD of Ukraine under grant $\Phi28.1/017$.}

\title[Hill's potentials in H\"{o}rmander spaces]
{Hill's potentials in H\"{o}rmander spaces and their spectral gaps$^{1}$}

\address{Institute of Mathematics of NAS of Ukraine \\
         Tereshchenkivska str., 3 \\
         Kyiv-4 \\
         Ukraine \\
         01601}

\author[V.Mikhailets, V. Molyboga] {Vladimir Mikhailets, Volodymyr Molyboga}

\email[Vladimir Mikhailets]{mikhailets@imath.kiev.ua}

\email[Volodymyr Molyboga]{molyboga@imath.kiev.ua}



%\title{Hill's potentials in H\"{o}rmander spaces and their spectral gaps\tnoteref{t1}} \tnotetext[t1]{The investigation is partially
%supported by DFFD of Ukraine under grant $\Phi28.1/017$.}

%\author[rvt]{Vladimir Mikhailets}
%\ead{mikhailets@imath.kiev.ua}

%\author[rvt]{Volodymyr Molyboga\corref{cor}}
%\ead{molyboga@imath.kiev.ua}

%\cortext[cor]{Corresponding author}

%\address[rvt]{Institute of Mathematics of NAS of Ukraine,
%Tereshchenkivska str. 3, Kyiv-4, Ukraine, 01601}

%\begin{keyword}
%Hill-Schr\"{o}dinger operators \sep singular potentials \sep
%spectral gaps \sep H\"{o}rmander spaces

%\MSC Primary 34L40 \sep Secondary 47A10 \sep 47A75
%\end{keyword}

%%%%%%%%%%%%%%%%%%%%%%%%%%%%%%%%%%%%%%%%%%%%%%%%%%%%%%%%%%%%%%%%%%%%%%%%%%%%%%%%%%%%%%%%%%%%%%%%%%%%%%%%%%%%%%%%%%%%%%%%%%%%
%%%%%%%%%%%%%%%%%%%%%%%%%%%%%%%%%%%%%%%%%%%%%%%%%%%%%%%%%%%%%%%%%%%%%%%%%%%%%%%%%%%%%%%%%%%%%%%%%%%%%%%%%%%%%%%%%%%%%%%%%%%%

\begin{abstract}
We study the behaviour of the lengths of spectral gaps $\{\gamma_{q}(n)\}_{n\in \mathbb{N}}$ in a continuous spectrum of the Hill-Schr\"{o}dinger
operators
\begin{equation*}
  S(q)u=-u''+q(x)u,\quad x\in \mathbb{R},
\end{equation*}
with 1-periodic real-valued distribution potentials
\begin{equation*}
  q(x)=\sum_{k\in \mathbb{Z}}\widehat{q}(k)\,e^{i k 2\pi x}\in H^{-1}(\mathbb{T}),\quad
  \widehat{q}(k)=\overline{\widehat{q}(-k)},\; k\in \mathbb{Z},
\end{equation*}
in dependence on the weight $\omega$ of the H\"{o}rmander space $H^{\omega}(\mathbb{T})\ni q$.

Let $h^{\omega}(\mathbb{N})$ be a Hilbert space of weighted sequences. We prove that
\begin{equation*}
\{\widehat{q}(\cdot)\}\in h^{\omega}(\mathbb{N})\Leftrightarrow\{\gamma_{q}(\cdot)\}\in h^{\omega}(\mathbb{N})
\end{equation*}
if a positive, in general non-monotonic, weight $\omega=\{\omega(k)\}_{k\in \mathbb{N}}$ satisfies only two following conditions:
\begin{align*}
  \verb"i")\hspace{5pt} & -1<\mu(\omega)=\liminf_{k\rightarrow\infty}\frac{\log\omega(k)}{\log k}\leq
  \rho(\omega)=\limsup_{k\rightarrow\infty}\frac{\log\omega(k)}{\log k}<\infty,\hspace{150pt} \\
 \verb" ii")\hspace{5pt} & \rho(\omega)<
  \begin{cases}
    1+2\mu(\omega) & \text{if}\hspace{15pt}\mu(\omega)\in(-1,0], \\
    1+\mu(\omega) & \text{if}\hspace{15pt}\mu(\omega)\in[0,\infty).
  \end{cases}
\end{align*}

In the case $q\in L^{2}(\mathbb{T})$, $\omega(k)=(1+2k)^{s}$, $s\in \mathbb{Z}_{+}$, this result is due to Marchenko and Os\-trov\-s\-kii~(1975).
\end{abstract}

\maketitle

%%%%%%%%%%%%%%%%%%%%%%%%%%%%%%%%%%%%%%%%%%%%%%%%%%%%%%%%%%%%%%%%%%%%%%%%%%%%%%%%%%%%%%%%%%%%%%%%%%%%%%%%%%%%%%%%%%%%%%%%%%%%
%%%%%%%%%%%%%%%%%%%%%%%%%%%%%%%%%%%%%%%%%%%%%%%%%%%%%%%%%%%%%%%%%%%%%%%%%%%%%%%%%%%%%%%%%%%%%%%%%%%%%%%%%%%%%%%%%%%%%%%%%%%%
\section{Introduction}\label{sct_Int}
Let consider on the complex Hilbert space $L^{2}(\mathbb{R})$ the Hill-Schr\"{o}dinger operators
\begin{equation}\label{eq_10}
  S(q)u:=-u''+q(x)u,\quad x\in \mathbb{R},
\end{equation}
with 1-periodic real-valued potentials
\begin{equation*}
  q(x)=\sum_{k\in \mathbb{Z}}\widehat{q}(k)e^{i k 2\pi x}\in L^{2}(\mathbb{T},\mathbb{R}),\quad
  \mathbb{T}:=\mathbb{R}/\mathbb{Z}.
\end{equation*}
This means that
\begin{equation*}\label{eq_11}
 \sum_{k\in \mathbb{Z}}|\widehat{q}(k)|^{2}<\infty
 \quad\text{and}\quad \widehat{q}(k)=\overline{\widehat{q}(-k)},\quad k\in \mathbb{Z}.
\end{equation*}

It is well known that the operators $S(q)$ are lower semibounded and self-adjoint. Their spectra are absolutely continuous and have a zone
structure \cite{ReSi4}.

Spectra of the operators $S(q)$, $q\in L^{2}(\mathbb{T},\mathbb{R})$, are completely defined by the location of the endpoints of spectral gaps
$\{\lambda_{0}(q),\lambda_{n}^{\pm}(q)\}_{n=1}^{\infty}$, which satisfy the inequalities:
\begin{equation}\label{InEq}
  -\infty<\lambda_{0}(q)<\lambda_{1}^{-}(q)\leq\lambda_{1}^{+}(q)<\lambda_{2}^{-}(q)\leq\lambda_{2}^{+}(q)<\cdots\,.
\end{equation}
For even/odd numbers $n\in \mathbb{Z}_{+}$ the endpoints of spectral gaps $\{\lambda_{0}(q),\lambda_{n}^{\pm}(q)\}_{n=1}^{\infty}$ are eigenvalues
of the periodic/semiperiodic problems on the interval $[0,1]$:
\begin{align*}
   S_{\pm}(q)u & :=-u''+q(x)u=\lambda u, \\
   \mathrm{Dom}(S_{\pm}(q)) & :=\left\{u\in H^{2}[0,1]\left|\, u^{(j)}(0)=\pm\, u^{(j)}(1),\, j=0,1\right.\right\}.\hspace{130pt}
\end{align*}

Interiors of spectral bands (stability or tied zones)
\begin{equation*}
  \mathcal{B}_{0}(q):=(\lambda_{0}(q),\lambda_{1}^{-}(q)),\qquad
  \mathcal{B}_{n}(q):=(\lambda_{n}^{+}(q),\lambda_{n+1}^{-}(q)),\quad n\in
  \mathbb{N},
\end{equation*}
together with the \textit{collapsed} gaps
\begin{equation*}
 \lambda=\lambda_{n}^{+}=\lambda_{n}^{-}, \quad n\in \mathbb{N}
\end{equation*}
are characterized as a locus of those real $\lambda\in \mathbb{R}$ for which all solutions of the equation $S(q) u=\lambda u$ are bounded.
\textit{Open} spectral gaps (instability or forbidden zones)
\begin{equation*}
  \mathcal{G}_{0}(q):=(-\infty,\lambda_{0}(q)),\qquad
  \mathcal{G}_{n}(q):=(\lambda_{n}^{-}(q),\lambda_{n}^{+}(q))\neq\emptyset,\quad n\in
  \mathbb{N}
\end{equation*}
are a locus of those real $\lambda\in \mathbb{R}$ for which any nontrivial solution of the equation $S(q) u=\lambda u$ is unbounded.

We are going to characterize the behaviour of the lengths of spectral gaps
\begin{equation*}
  \gamma_{q}(n):=\lambda_{n}^{+}(q)-\lambda_{n}^{-}(q),\quad n\in \mathbb{N}
\end{equation*}
of the Hill-Schr\"{o}dinger operators $S(q)$ in terms of the behaviour of the Fourier coefficients $\{\widehat{q}(n)\}_{n\in \mathbb{N}}$ of the
potentials $q$ with respect to appropriate weight spaces, that is by means of potential regularity.

For $L^{2}(\mathbb{T},\mathbb{R})$-potentials fundamental result in this problem follows from the Marchenko and Ostrovskii paper~\cite{MrOs} (see
also \cite{Mrch}):
\begin{equation}\label{eq_14}
  q\in H^{s}(\mathbb{T},\mathbb{R})\Leftrightarrow\sum_{n\in \mathbb{N}}(1+2n)^{2s}\gamma_{q}^{2}(n),\qquad s\in
  \mathbb{Z}_{+},
\end{equation}
where $H^{s}(\mathbb{T},\mathbb{R})$, $s\in \mathbb{Z}_{+}$, denotes the Sobolev spaces of 1-periodic real-valued functions on the circle
$\mathbb{T}$.

To characterize regularity of potentials in the finer way we will use the real H\"{o}rmander spaces
\begin{equation*}
  H^{\omega}(\mathbb{T},\mathbb{R}):=\left\{f=\sum_{k\in \mathbb{Z}}\widehat{f}\,(k)e^{i k2\pi
  x}\left|\;\widehat{f}(k)=\overline{\widehat{f}(-k)},\;k\in \mathbb{Z},\;
  \sum_{k\in \mathbb{N}}  \omega^{2}(k)|\widehat{f}(k)|^{2}<\infty\right.\right\},
\end{equation*}
where $\omega(\cdot)$ is a positive weight. In the case of the Sobolev spaces it is a power one. Such definition of the real H\"{o}rmander spaces
on the circle completely corresponds to the theory of function spaces on a smooth closed manifold~\cite{MiMr1, MiMr2}.

Djakov, Mityagin \cite{DjMt2}, P\"{o}schel \cite{Psch1, Psch2} extended the Marchenko-Ostrovskii Theorem \eqref{eq_14} to the general class of
weights $\Omega=\{\Omega(k)\}_{k\in \mathbb{N}}$ satisfying the following conditions:
\begin{align*}
 \verb"i")\hspace{5pt} & \Omega(k)\nearrow\infty,\; k\in \mathbb{N};\hspace{5pt}\text{(monotonicity)}  \hspace{220pt} \\
\verb"ii")\hspace{5pt} & \Omega(k+m)\leq \Omega(k)\Omega(m)\quad \forall k,m\in \mathbb{N};\hspace{5pt}\text{(submultiplicity)} \\
\verb"iii")\hspace{5pt} & \frac{\log\Omega(k)}{k}\searrow 0,\quad k\rightarrow\infty,\hspace{5pt}\text{(subexponentiality)}.
\end{align*}
For such weights they proved that
\begin{equation}\label{eq_16}
  q\in H^{\Omega}(\mathbb{T},\mathbb{R})\Leftrightarrow \{\gamma_{q}(\cdot)\}\in  h^{\Omega}(\mathbb{N}).
\end{equation}
Here $h^{\Omega}(\mathbb{N})$ is the Hilbert space of weighted sequences generated by the weight $\Omega(\cdot)$.

Earlier Kappeler, Mityagin \cite{KpMt2} proved the direct implication in \eqref{eq_16} under the only assumption of submultiplicity. In the
special cases of the Abel-Sobolev weights, the Gevrey weights and the slowly increasing weights the relationship \eqref{eq_16} was established by
Kappeler, Mityagin \cite{KpMt1} ($\Rightarrow$) and Djakov, Mityagin \cite{DjMt, DjMt1} ($\Leftarrow$). Detailed exposition of these results is
given in the survey \cite{DjMt2}. It should be noted that P\"{o}schel \cite{Psch1, Psch2}, Kappeler, Mityagin \cite{KpMt1, KpMt2} and Djakov,
Mityagin \cite{DjMt1, DjMt2} studied also the more general case of complex-valued potentials.

For the power weights
\begin{equation*}
 w_{s}=\left\{w_{s}(k)\right\}_{k\in \mathbb{N}}:\qquad w_{s}(k):=(1+2k)^{s},\quad s\in \mathbb{R}
\end{equation*}
it is convenient to use shorter notation
\begin{equation*}
  H^{\omega_{s}}(\mathbb{T})\equiv H^{s}(\mathbb{T}),\quad h^{\omega_{s}}(\mathbb{N})\equiv h^{s}(\mathbb{N}).
\end{equation*}

After the celebrated Kronig and Penney paper \cite{KrPn} the Schr\"{o}dinger operators with (periodic) distributions as potentials came into
mathematical physics. The following development of quantum mechanics stimulated active growth of this branch of science (see the bibliography of
the monographs \cite{AlGHH, AlKr}, which counts several hundreds of physical and mathematical works).

In this paper we study the Hill-Schr\"{o}dinger operators $S(q)$ with 1-periodic real-valued distribution potentials $q$ from the negative Sobolev
space $H^{-1}(\mathbb{T},\mathbb{R})$:
\begin{equation}\label{eq_MCP}
  q(x)=\sum_{k\in \mathbb{Z}}\widehat{q}(k)e^{i k 2\pi x}\in H^{-1}(\mathbb{T},\mathbb{R}).
\end{equation}
This means that
\begin{equation*}
    \sum_{k\in \mathbb{N}}(1+2k)^{2s}|\widehat{q}(k)|^{2}<\infty,\quad s=-1,
    \quad\text{and}\quad  \widehat{q}(k)=\overline{\widehat{q}(-k)},
    \quad k\in \mathbb{Z}.   \eqno(6_{s})
\end{equation*}
All pseudo-functions, measures, pseudo-measures and some even more singular distributions on the circle satisfy this condition.

Under the assumption \eqref{eq_MCP} the operators \eqref{eq_10} can be well defined on the complex Hilbert space $L^{2}(\mathbb{R})$ in the
following basic ways:
\begin{itemize}
  \item as form-sum operators;
  \item as quasi-differential operators (minimal operators, maximal operators, the Friedrichs extensions of the minimal
  operators);
  \item as limits of operators with smooth 1-periodic potentials in the norm resolvent sense.
\end{itemize}
Equivalence of all these definitions was proved in the paper \cite{MiMl6}.
%%%%%%%%%%%%%%%%%%%%%%%%%%%%%%%%%%%%%%%%%%%%%%%%%%%%%%%%%%%%%%%%%%%%%%%%%%%%%%%%%%%%%%%%%%%%%%%%%%%%%%%%%%%%%%%%%%%%%%%%%%%%
\setcounter{equation}{6}
%%%%%%%%%%%%%%%%%%%%%%%%%%%%%%%%%%%%%%%%%%%%%%%%%%%%%%%%%%%%%%%%%%%%%%%%%%%%%%%%%%%%%%%%%%%%%%%%%%%%%%%%%%%%%%%%%%%%%%%%%%%%

The Hill-Schr\"{o}dinger operators $S(q)$, $q\in H^{-1}(\mathbb{T},\mathbb{R})$, are lower semibounded and self-adjoint, their spectra are
absolutely continuous and have a band and gap structure as in the classical case of $L^{2}(\mathbb{T},\mathbb{R})$-potentials, \cite{HrMk, Krt,
DjMt3, MiMl6, MiSb, GsKr}. The endpoints of spectral gaps $\{\lambda_{0}(q),\lambda_{n}^{\pm}(q)\}_{n=1}^{\infty}$ satisfy the inequalities
\eqref{InEq}. For even/odd numbers $n\in \mathbb{Z}_{+}$ they are eigenvalues of the periodic/semiperiodic problems on the interval $[0,1]$
\cite[Theorem~C]{MiMl6}.

In the paper \cite{MiMl7} we extended the Marchenko-Ostrovskii Theorem \eqref{eq_14} to the case of singular potentials $q\in
H^{-1+}(\mathbb{T},\mathbb{R})$. This means that $q$ satisfies $(6_{s})$ with some $s>-1$. We proved that
\begin{equation}\label{eq_18}
  q\in H^{s}(\mathbb{T},\mathbb{R})\Leftrightarrow \{\gamma_{q}(\cdot)\}\in h^{s}(\mathbb{N}),\qquad s\in (-1,\infty).
\end{equation}
The case $s\in (-1,0]$ was earlier treated in \cite{MiMl3, Mlb2}.

Djakov, Mityagin \cite{DjMt4, DjMt5} extended the latter statement to the limiting case $s=-1$:
\begin{equation}\label{eq_19.1}
  q\in H^{s}(\mathbb{T},\mathbb{R})\Leftrightarrow \{\gamma_{q}(\cdot)\}\in h^{s}(\mathbb{N}),\qquad s\in [-1,\infty),
\end{equation}
under the \textit{a priori} assumption $q\in H^{-1}(\mathbb{T},\mathbb{R})$. Moreover, they extended the result \eqref{eq_16} to the case of
potentials $q\in H^{-1}(\mathbb{T},\mathbb{R})$ and more general weights:
\begin{equation}\label{eq_19.2}
  q\in H^{\Omega^{\ast}}(\mathbb{T},\mathbb{R})\Leftrightarrow \{\gamma_{q}(\cdot)\}\in  h^{\Omega^{\ast}}(\mathbb{N}),
  \qquad \Omega^{*}:=\left\{\frac{\Omega(k)}{1+2k}\right\}_{k\in \mathbb{N}},
\end{equation}
where the weights $\Omega=\{\Omega(k)\}_{k\in \mathbb{N}}$ are supposed to be monotonic, submultiplicative and subexponential.
%%%%%%%%%%%%%%%%%%%%%%%%%%%%%%%%%%%%%%%%%%%%%%%%%%%%%%%%%%%%%%%%%%%%%%%%%%%%%%%%%%%%%%%%%%%%%%%%%%%%%%%%%%%%%%%%%%%%%%%%%%%%
%%%%%%%%%%%%%%%%%%%%%%%%%%%%%%%%%%%%%%%%%%%%%%%%%%%%%%%%%%%%%%%%%%%%%%%%%%%%%%%%%%%%%%%%%%%%%%%%%%%%%%%%%%%%%%%%%%%%%%%%%%%%
\section{Main results}\label{sct_MnRs}
The aim of this paper is to extend the result \eqref{eq_19.1} to a more extensive class of weights, for which the conditions of regularity of the
weight behaviour may not hold, and to supplement the result \eqref{eq_19.2}.

For convenience of formulation of the results we introduce the following definition.
\begin{definition*}\label{df_10}
Let the set $X\subset H^{-1}(\mathbb{T},\mathbb{R})$. We write $\omega\in \mathrm{MO}(X)$ if
\begin{equation*}
  q\in H^{\omega}(\mathbb{T},\mathbb{R})\Leftrightarrow \{\gamma_{q}(\cdot)\}\in
  h^{\omega}(\mathbb{N})\quad \forall q\in X.
\end{equation*}
\end{definition*}
\noindent It is easy to see that
\begin{equation*}
 X \supset Y  \Rightarrow \mathrm{MO}(X) \subset \mathrm{MO}(Y),\hspace{50pt}
\end{equation*}
and
\begin{align*}
  \eqref{eq_14} & \Leftrightarrow \omega_{s} \in \mathrm{MO}(L^{2}(\mathbb{T})),\quad s\in \mathbb{Z}_{+}, \\
  \eqref{eq_16} & \Leftrightarrow \Omega\in \mathrm{MO}(L^{2}(\mathbb{T})), \\
  \eqref{eq_19.1} & \Leftrightarrow \omega_{s}\in \mathrm{MO}(H^{-1}(\mathbb{T})),\quad  s\in  [-1,\infty), \\
  \eqref{eq_19.2} & \Leftrightarrow \Omega^{\ast}\in \mathrm{MO}(H^{-1}(\mathbb{T})).
\end{align*}

Further, let us recall that \textit{lower order} $\mu(\omega)$ and \textit{upper order} $\rho(\omega)$ of a weight sequence
$\omega=\{\omega(k)\}_{k\in\mathbb{N}}$ are defined as follows:
\begin{equation*}
  \mu\equiv\mu(\omega):=\liminf_{k\rightarrow\infty}\frac{\log\omega(k)}{\log k},\qquad
  \rho\equiv\rho(\omega):=\limsup_{k\rightarrow\infty}\frac{\log\omega(k)}{\log k}.
\end{equation*}

In what follows we will consider only weights satisfying the conditions:
\begin{equation*}
  -1<\mu(\omega)\leq\rho(\omega)<\infty.
\end{equation*}
Such weights have not more than a power growth. The following statement is the main result of this paper.
\begin{theorem}\label{th_10}
Let $q\in H^{-1}(\mathbb{T},\mathbb{R})$ and the weight $\omega=\{\omega(k)\}_{k\in\mathbb{N}}$ satisfy conditions:
\begin{align}
   & -1<\mu(\omega)\leq\rho(\omega)<\infty, \label{eq_20} \\
   & \rho(\omega)<
   \begin{cases}
    1+2\mu(\omega) & \text{if}\hspace{15pt}\mu(\omega)\in(-1,0],  \\
    1+\mu(\omega) & \text{if}\hspace{15pt}\mu(\omega)\in[0,\infty).
  \end{cases}
  \label{eq_22}
\end{align}
Then $\omega\in\mathrm{MO}(H^{-1}(\mathbb{T}))$.
\end{theorem}
\begin{corollary}\label{cr_10}
Let for the weight $\omega=\{\omega(k)\}_{k\in\mathbb{N}}$ exist the order
\begin{equation*}\label{eq_24}
  \lim_{k\rightarrow\infty}\frac{\log\omega(k)}{\log k}=s\in(-1,\infty).
\end{equation*}
Then $\omega\in\mathrm{MO}(H^{-1}(\mathbb{T}))$.
\end{corollary}

From Corollary \ref{cr_10} we receive the following result.
\begin{corollary}[cf. \cite{MiMl7}]\label{cr_12}
Let the weight $\omega=\{\omega(k)\}_{k\in\mathbb{N}}$ be a regular varying sequence on $+\infty$ in the Karamata sense with the index
$s\in(-1,\infty)$. Then $\omega\in\mathrm{MO}(H^{-1}(\mathbb{T}))$.
\end{corollary}

Note that the assumption of Corollary \ref{cr_12} holds, for instance, for the weight
\begin{align*}
  \omega(k) & =(1+2k)^{s}\,(\log (1+k))^{r_{1}}(\log\log (1+k))^{r_{2}}\ldots (\log\log\ldots\log (1+k))^{r_{p}}, \\
    s & \in(-1,\infty),\;\{r_{1},\ldots,r_{p}\}\subset \mathbb{R},\;p\in \mathbb{N},\hspace{210pt}
\end{align*}
see the monographs \cite{BnGlTg, Snt}.

Theorem \ref{th_10} extends the statement \eqref{eq_19.1} and Corollary \ref{cr_12} to the case of non-regularly varying weights.

The following Example A shows that statement \eqref{eq_19.2} does not cover Corollary \ref{cr_10} and all the more Theorem~\ref{th_10}.
\begin{example}\label{ex_10}
Let $s\in (-1,\infty)$. Set
\begin{equation*}
  w(k):=
  \begin{cases}
    k^{s}\log(1+k) & \text{if}\quad k\in 2\mathbb{N}, \\
    k^{s} & \text{if}\quad k\in (2\mathbb{N}-1).
  \end{cases}
\end{equation*}
\end{example}
\noindent Then the weight $\omega=\{\omega(k)\}_{k\in \mathbb{N}}$ satisfies the conditions of Corollary \ref{cr_10}. But one can prove that the
weight
\begin{equation*}
  \omega^{\ast}:=\{(1+2k)\,\omega(k)\}_{k\in \mathbb{N}}
\end{equation*}
is not equivalent to any monotonic weight.
%
%\begin{example}\label{ex_12}
%Let
%\begin{align*}
%  \verb"i")\hspace{5pt} & s\in (-1,\infty),\hspace{380pt} \\
%  \verb"ii")\hspace{5pt} & 0< a <
%  \begin{cases}
%    1+s & \text{if}\quad s\in (-1,0], \\
%    1 & \text{if}\quad  s\in [0,\infty).
%  \end{cases}
%  \\
%  \verb"iii")\hspace{5pt} & \mathbb{N}:=\mathbb{K}_{1}\bigcup\mathbb{K}_{2},\quad \sharp\mathbb{K}_{1}=\aleph_{0},\;
%  \sharp\mathbb{K}_{2}=\aleph_{0},\;\mathbb{K}_{1}\bigcap\mathbb{K}_{2}=\emptyset.
%\end{align*}
%Define the weight $\omega$ on $\mathbb{N}$ as:
%\begin{equation*}
%\omega=\{\omega(k)\}_{k\in \mathbb{N}}:\qquad  w(k):=
%  \begin{cases}
%    k^{s}\log(1+k) & \text{if}\quad k\in \mathbb{K}_{1}, \\
%    k^{s+a}\log^{-1}(1+k) & \text{if}\quad k\in \mathbb{K}_{2}.
%  \end{cases}
%\end{equation*}
%\end{example}
%%%%%%%%%%%%%%%%%%%%%%%%%%%%%%%%%%%%%%%%%%%%%%%%%%%%%%%%%%%%%%%%%%%%%%%%%%%%%%%%%%%%%%%%%%%%%%%%%%%%%%%%%%%%%%%%%%%%%%%%%%%%
%%%%%%%%%%%%%%%%%%%%%%%%%%%%%%%%%%%%%%%%%%%%%%%%%%%%%%%%%%%%%%%%%%%%%%%%%%%%%%%%%%%%%%%%%%%%%%%%%%%%%%%%%%%%%%%%%%%%%%%%%%%%
\section{Preliminaries}\label{sct_Prl}
Here, for convenience, we define Hilbert spaces of weighted two-sided sequences and formulate the Convolution Lemma \ref{lm_10}.

For every positive sequence $\omega=\{\omega(k)\}_{k\in\mathbb{N}}$ there exists its unique extension on $\mathbb{Z}$ which is a two-sided
sequence satisfying the conditions:
\begin{align*}
  \verb"i")\hspace{5pt} & \omega(0)=1;\hspace{340pt} \\
  \verb"ii")\hspace{5pt} & \omega(-k)=\omega(k)\quad \forall k\in \mathbb{N}; \\
  \verb"iii")\hspace{5pt} & \omega(k)>0\quad \forall k\in \mathbb{Z}.
\end{align*}

Let
\begin{equation*}
  h^{\omega}(\mathbb{Z})\equiv h^{\omega}(\mathbb{Z},\mathbb{C})
\end{equation*}
be the Hilbert space of two-sided sequences:
\begin{align*}
  h^{\omega}(\mathbb{Z}) & :=\left\{a=\{a(k)\}_{k\in \mathbb{Z}}\left|
 \sum_{k\in \mathbb{Z}}\omega^{2}(k)|a(k)|^{2}<\infty\right.\right\}, \\
  (a,b)_{h^{\omega}(\mathbb{Z})} & :=\sum_{k\in \mathbb{Z}}\omega^{2}(k)a(k)\overline{b(k)},\quad
  a,b\in h^{\omega}(\mathbb{Z}), \\
  \|a\|_{h^{\omega}(\mathbb{Z})} & :=(a,a)_{h^{\omega}(\mathbb{Z})}^{1/2},\quad
  a\in h^{\omega}(\mathbb{Z}).
\end{align*}

Basic weights which we use are the power ones:
\begin{equation*}
 w_{s}=\left\{w_{s}(k)\right\}_{k\in \mathbb{Z}}:\qquad w_{s}(k)=(1+2|k|)^{s},\quad s\in
 \mathbb{R}.
\end{equation*}
In this case it is convenient to use shorter notations:
\begin{equation*}
   h^{\omega_{s}}(\mathbb{Z})\equiv h^{s}(\mathbb{Z}),\quad s\in \mathbb{R}.
\end{equation*}

Operation of convolution for two-sided sequences
\begin{equation*}
  a=\{a(k)\}_{k\in\mathbb{Z}}\quad\text{and}\quad b=\{b(k)\}_{k\in\mathbb{Z}}
\end{equation*}
is formally defined as follows:
\begin{align*}
  (a,b) & \mapsto a\ast b, \\
  (a\ast b)(k) & :=\sum_{j\in \mathbb{Z}}a(k-j)\,b(j),\quad k\in \mathbb{Z}.
\end{align*}

Sufficient conditions for the convolution to exist as a continuous map are given by the following known lemma, see for example \cite{KpMh, Mhr}.
\begin{lemma}[The Convolution Lemma]\label{lm_10}
Let $s,r\geq 0$, and $t\leq\min(s,r)$, $t\in \mathbb{R}$. If $s+r-t>1/2$ then the convolution $(a,b)\mapsto a\ast b$ is well defined as a
continuous map acting in the spaces:
\begin{align*}
  (a)\hspace{5pt} & h^{s}(\mathbb{Z})\times h^{r}(\mathbb{Z})\rightarrow h^{t}(\mathbb{Z}),\hspace{280pt} \\
  (b)\hspace{5pt} & h^{-t}(\mathbb{Z})\times h^{s}(\mathbb{Z})\rightarrow h^{-r}(\mathbb{Z}).
\end{align*}

In the case $s+r-t<1/2$ this statement fails to hold.
\end{lemma}
%%%%%%%%%%%%%%%%%%%%%%%%%%%%%%%%%%%%%%%%%%%%%%%%%%%%%%%%%%%%%%%%%%%%%%%%%%%%%%%%%%%%%%%%%%%%%%%%%%%%%%%%%%%%%%%%%%%%%%%%%%%%
%%%%%%%%%%%%%%%%%%%%%%%%%%%%%%%%%%%%%%%%%%%%%%%%%%%%%%%%%%%%%%%%%%%%%%%%%%%%%%%%%%%%%%%%%%%%%%%%%%%%%%%%%%%%%%%%%%%%%%%%%%%%
\section{The Proofs}\label{sct_Prf}
Basic point of our proof of Theorem \ref{th_10} is sharp asymptotic formulae for the lengths of spectral gaps $\{\gamma_{q}(n)\}_{n\in\mathbb{N}}$
of the Hill-Schr\"{o}dinger operators $S(q)$.
\begin{lemma}\label{lm_12}
The lengths of spectral gaps $\{\gamma_{q}(n)\}_{n\in\mathbb{N}}$ of the Hill-Schr\"{o}dinger operators $S(q)$ with $q\in
H^{s}(\mathbb{T},\mathbb{R})$, $s\in (-1,\infty)$, uniformly on the bounded sets of potentials $q$ in the corresponding Sobolev spaces
$H^{s}(\mathbb{T})$ for $n\geq n_{0}$, $n_{0}=n_{0}\left(\|q\|_{H^{s}(\mathbb{T})}\right)$, satisfy the following asymptotic formulae:
\begin{align}
  \gamma_{q}(n) & =2|\widehat{q}(n)|+h^{1+2s-\varepsilon}(n),\;\varepsilon>0 & & \text{if}\quad s\in(-1,0],\label{eq_28} \\
  \gamma_{q}(n) & =2|\widehat{q}(n)|+h^{1+s}(n) & & \text{if}\quad s\in[0,\infty).\label{eq_30}
\end{align}
\end{lemma}
\begin{proof}[Proof of Lemma \ref{lm_12}]
The asymptotic estimates \eqref{eq_28} were established by the authors in the paper \cite[Theorem~1]{MiMl7} (see also \cite{MiMl3, Mlb2}) by
method of the isospectral transformation of the problem \cite{Mlb1, MiMl1, MiMl2, MiMl3, Mlb2}.

The asymptotic formulae \eqref{eq_30} follow from \cite[Theorem 1.2]{KpMt2} due to the Convolution Lemma~\ref{lm_10} (see also
\cite[Appendix]{KpMt2}). Indeed, from \cite[Theorem 1.2]{KpMt2} with $q\in H^{s}(\mathbb{T},\mathbb{R})$, $s\in [0,\infty)$, we get
\begin{equation}\label{eq_32}
  \sum_{n\in \mathbb{N}}(1+2n)^{2(1+s)}\left(\min_{\pm}\left|\gamma_{q}(n)\pm
  2\sqrt{(\widehat{q}+\varrho)(-n)(\widehat{q}+\varrho)(n)}\right|\right)^{2}\leq
  C\left(\|q\|_{H^{s}(\mathbb{T})}\right),
\end{equation}
where
\begin{equation*}
  \varrho(n):=\frac{1}{\pi^{2}}\sum_{j\in
  \mathbb{Z}\setminus\{\pm n\}}\frac{\widehat{q}(n-j)\widehat{q}(n+j)}{(n-j)(n+j)}.
\end{equation*}
Without losing generality we assume that
\begin{equation}\label{eq_33}
 \widehat{q}(0):=0.
\end{equation}

Taking into account that the potentials $q$ are real-valued we have
\begin{equation*}
  \widehat{q}(k)=\overline{\widehat{q}(-k)},\; \varrho(k)=\overline{\varrho(-k)},\quad  k\in \mathbb{Z}.
\end{equation*}
Then from \eqref{eq_32} we get the estimates
\begin{equation}\label{eq_34}
  \left\{\gamma_{n}(q)-2\left|\widehat{q}(n)+\varrho(n)\right|\right\}_{n\in \mathbb{N}}\in h^{1+s}(\mathbb{N}).
\end{equation}

Further, as  by assumption $q\in H^{s}(\mathbb{T},\mathbb{R})$, that is $\{\widehat{q}(k)\}_{k\in \mathbb{Z}}\in h^{s}(\mathbb{Z})$, then
\begin{equation*}
  \left\{\frac{\widehat{q}(k)}{k}\right\}_{k\in\mathbb{Z}}\in h^{1+s}(\mathbb{Z}),\quad s\in
  [0,\infty),
\end{equation*}
taking into account \eqref{eq_33}. Applying the Convolution Lemma \ref{lm_10} we obtain
\begin{align}\label{eq_36}
  \varrho(n) & =\frac{1}{\pi^{2}}\sum_{j\in
  \mathbb{Z}}\frac{\widehat{q}(n-j)\widehat{q}(n+j)}{(n-j)(n+j)}=\frac{1}{\pi^{2}}\sum_{j\in
  \mathbb{Z}}\frac{\widehat{q}(2n-j)}{2n-j}\cdot\frac{\widehat{q}(j)}{j}  \\
  & =\left(\left\{\frac{\widehat{q}(k)}{k}\right\}_{k\in \mathbb{Z}}\ast\left\{\frac{\widehat{q}(k)}{k}\right\}_{k\in
  \mathbb{Z}}\right)(2n)\in h^{1+s}(\mathbb{N}). \notag
\end{align}

Finally, from \eqref{eq_34} and \eqref{eq_36} we get the necessary estimates \eqref{eq_30}.

The proof of Lemma \ref{lm_12} is complete.
\end{proof}

\textbf{Proof of Theorem \ref{th_10}}. Let  $q\in H^{-1}(\mathbb{T},\mathbb{R})$ and $\omega=\{\omega(k)\}_{k\in\mathbb{N}}$ be a given weight
satisfying the conditions \eqref{eq_20} and \eqref{eq_22} of Theorem \ref{th_10}. We need to prove the statement
\begin{equation}\label{eq_38}
  q\in H^{\omega}(\mathbb{T},\mathbb{R})\Leftrightarrow \{\gamma_{q}(\cdot)\}\in  h^{\omega}(\mathbb{N}).
\end{equation}

From formula \eqref{eq_20} and definition of the lower and upper orders of a weight sequence we conclude that for the given weight
$\omega=\{\omega(k)\}_{k\in\mathbb{N}}$  the following estimates are fulfilled:
\begin{equation*}
   k^{\mu-\delta}\ll\omega(k)\ll k^{\rho+\delta},\qquad -1<\mu\leq\rho<\infty,\;\delta>0.
\end{equation*}
Hence, the continuous embeddings
\begin{align}
 H^{\rho+\delta}(\mathbb{T}) & \hookrightarrow H^{\omega}(\mathbb{T})\hookrightarrow H^{\mu-\delta}(\mathbb{T}),\label{eq_40} \\
 h^{\rho+\delta}(\mathbb{N}) & \hookrightarrow h^{\omega}(\mathbb{N})\hookrightarrow h^{\mu-\delta}(\mathbb{N}), \hspace{20pt}
 -1<\mu\leq\rho<\infty,\;\delta>0 \label{eq_42}
\end{align}
are valid becouse of
\begin{equation}\label{eq_44}
  H^{\omega_{1}}(\mathbb{T})\hookrightarrow H^{\omega_{2}}(\mathbb{T}),\quad
  h^{\omega_{1}}(\mathbb{N})\hookrightarrow h^{\omega_{2}}(\mathbb{N})\quad\text{if only}\quad \omega_{1}\gg \omega_{2}.
\end{equation}

Let $q\in H^{\omega}(\mathbb{T},\mathbb{R})$, then from \eqref{eq_40} we obtain that $q\in H^{\mu-\delta}(\mathbb{T},\mathbb{R})$, $\delta>0$. Due
to arbitrary choice of $\delta>0$ we may choose it so that:
\begin{align*}
  & \mu-\delta>0 \quad\text{if}\quad \mu>0, \\
  & \mu-\delta>-1 \quad\text{otherwise}.
\end{align*}
Further, using Lemma \ref{lm_12} we get the following asymptotic formulae for the lengths of spectral gaps:
\begin{align}
  \gamma_{q}(n) & =2|\widehat{q}(n)|+h^{1+2(\mu-\delta)-\varepsilon}(n),\;\varepsilon>0 & & \text{if}\quad
  \mu-\delta\in(-1,0],\label{eq_46} \\
  \gamma_{q}(n) & =2|\widehat{q}(n)|+h^{1+\mu-\delta}(n) & & \text{if}\quad \mu-\delta\in[0,\infty).\label{eq_48}
\end{align}
Now, due to possibility of arbitrary choice of $\delta>0$ and $\varepsilon>0$ we may choose them so that:
\begin{align}
  1+2(\mu-\delta)-\varepsilon>\rho+\delta & \hspace{10pt}\text{if}\hspace{10pt}\mu-\delta\in (-1,0],\label{eq_50} \\
  1+\mu-\delta>\rho+\delta & \hspace{10pt}\text{if}\hspace{10pt}\mu-\delta\in [0,\infty).\label{eq_52}
\end{align}
This choice is possible as \eqref{eq_20} and \eqref{eq_22} hold.

Taking into account \eqref{eq_50}, \eqref{eq_52} and using \eqref{eq_44}, from \eqref{eq_46}, \eqref{eq_48} we get estimates
\begin{equation*}
  \gamma_{q}(n)=2|\widehat{q}(n)|+h^{\rho+\delta}(n).
\end{equation*}
From the latter formula and formula \eqref{eq_42} we finally derive
\begin{equation}\label{eq_54}
   \gamma_{q}(n)=2|\widehat{q}(n)|+h^{\omega}(n).
\end{equation}
Hence, as by assumption $q\in H^{\omega}(\mathbb{T},\mathbb{R})$, and as consequence $\{\widehat{q}(\cdot)\}\in h^{\omega}(\mathbb{N})$, we get
$\{\gamma_{q}(\cdot)\}\in h^{\omega}(\mathbb{N})$.

The implication $(\Rightarrow)$ in \eqref{eq_38} has been proved.

Conversely, let $\{\gamma_{q}(\cdot)\}\in h^{\omega}(\mathbb{N})$. Then applying \eqref{eq_42} we have $\{\gamma_{q}(\cdot)\}\in
h^{\mu-\delta}(\mathbb{N})$, $\delta>0$:
\begin{align*}
  & \mu-\delta>0 \quad\text{if}\quad \mu>0, \\
  & \mu-\delta>-1 \quad\text{otherwise}.
\end{align*}
Now, applying \eqref{eq_19.1} we conclude that $q\in H^{\mu-\delta}(\mathbb{T},\mathbb{R})$.

We have already proved the implication
\begin{equation*}
  q\in H^{\mu-\delta}(\mathbb{T},\mathbb{R})\Rightarrow \eqref{eq_54}.
\end{equation*}
So we have
\begin{equation*}
  \gamma_{q}(n)=2|\widehat{q}(n)|+h^{\omega}(n),
\end{equation*}
and hence $\{\widehat{q}(\cdot)\}\in h^{\omega}(\mathbb{N})$, i.e., $q\in H^{\omega}(\mathbb{T},\mathbb{R})$.

The implication $(\Leftarrow)$ in \eqref{eq_38} has been proved.

The proof of Theorem \ref{th_10} is complete.
%%%%%%%%%%%%%%%%%%%%%%%%%%%%%%%%%%%%%%%%%%%%%%%%%%%%%%%%%%%%%%%%%%%%%%%%%%%%%%%%%%%%%%%%%%%%%%%%%%%%%%%%%%%%%%%%%%%%%%%%%%%%
\section*{Acknowledgments}
The authors would like to express their gratitude to Professors V.A.~Mar\-chen\-ko and F.S.~Rofe-Beketov, for their attention to this paper.

%%%%%%%%%%%%%%%%%%%%%%%%%%%%%%%%%%%%%%%%%%%%%%%%%%%%%%%%%%%%%%%%%%%%%%%%%%%%%%%%%%%%%%%%%%%%%%%%%%%%%%%%%%%%%%%%%%%%%%%%%%%%
%\newpage
%%%%%%%%%%%%%%%%%%%%%%%%%%%%%%%%%%%%%%%%%%%%%%%%%%%%%%%%%%%%%%%%%%%%%%%%%%%%%%%%%%%%%%%%%%%%%%%%%%%%%%%%%%%%%%%%%%%%%%%%%%%%

\end{document}